\documentclass[12pt,a4paper,oneside,notitlepage]{article}
\usepackage[utf8]{inputenc}            
\usepackage[T1]{fontenc}

\usepackage{amsmath}             
\usepackage{amssymb}             
\usepackage{amsfonts}

\usepackage{latexsym}            
\usepackage{amsthm}             
\usepackage{cases}                       
\usepackage{mathrsfs}
\usepackage{xcolor}

\newcommand{\field}[1]{\mathbb{#1}}   
\newcommand{\R}{\field{R}}            
\newcommand{\s}{\field{S}}

\newcommand{\I}{\int_{\R^n}}
\newcommand{\rmi}{\mathrm{i}}
\newcommand{\e}{\mathrm{e}}
\newcommand{\G}{G_k^+}
\newcommand{\usc}{u_{\mathrm{sc}}}
\renewcommand{\d}{\mathrm{d}}
\newcommand{\Is}{\int_{\field{S}^{n-1}}}

\theoremstyle{plain}
\newtheorem{theorem}{Theorem}[section]
\newtheorem{lemma}[theorem]{Lemma}

\newtheorem{assumption}[theorem]{Assumption}
\theoremstyle{definition}
\newtheorem{definition}[theorem]{Definition}

\theoremstyle{remark}
\newtheorem*{remark}{Remark}

\title{Recovery of singularities from fixed angle scattering data for biharmonic operator in dimensions two and three}

\author{Jaakko Kultima\footnote{corresponding author, Research Unit of Mathematical Sciences, P.O. BOX 3000, FIN-90014 University of Oulu}}

\begin{document}
\maketitle

\begin{abstract}
The inverse fixed angle problem for operator $\Delta^2 u + V(x,|u|) u$ is considered in dimensions $n=2,3$. We prove that the difference between an inverse fixed angle Born approximation and the function $V(\cdot,1)$ is smoother than the function $V$ itself in some Sobolev scale. This allows us to conclude that the main singularities of the perturbation $V$ can be reconstructed from the knowledge of the scattering amplitude with some fixed incident angle.
\end{abstract}

\section{Introduction}
We consider the following n-dimensional ($n=2,3$) biharmonic differential operator
$$H_4 u(x) = \Delta^2 u(x) + V(x,|u|) u(x),$$
where $\Delta^2$ is the bi-Laplacian and it is perturbed by zero-order, complex-valued function $V$, which depends on both the spatial variable $x\in\R^n$ and non-linearly on the absolute value of the argument function $u$. The scattering problem is formulated as 
\begin{equation*} 
\begin{cases} 
& H_4 u(x,k,\theta) = k^4 u(x,k,\theta),\\
& u(x,k,\theta) = u_0(x,k,\theta) + \usc(x,k,\theta),
\end{cases}
\end{equation*}
where the real-number $k>0$ is called a wavenumber and it is inversely proportional to the wavelength. The solution (total scattering field) to the differential equation is assumed to be a superposition between an incoming plane wave $u_0(x,k,\theta) = \e^{\rmi k(\theta,x)}$, coming from direction $\theta\in\s^{n-1}$, and an outgoing scattered wave $\usc$. The fact that the scattered wave is outgoing is formalised by assuming that it satisfies the Sommerfeld radiation conditions 
$$\lim_{r\to\infty} r^{\frac{n-1}{2}} \left(\frac{\partial}{\partial r} f - \rmi k f \right) = 0,$$
where $r=|x|$, for both $f=\usc$ and $f=\Delta \usc$. These radiation conditions guarantee the uniqueness for the solution of the scattering problem.
In \cite{TS2018}, this direct scattering problem was discussed in the case of a linear perturbation, and it was proved that sufficiently well-behaved solution to the original differential equation is also a solution to the Lippmann-Schwinger integral equation. Similar arguments can be applied in the non-linear case and we may reformulate the original problem as solving the following integral equation
\begin{equation}\label{LS}
    \usc(x) = -\I \G(|x-y|) V(y,|u_0+\usc|)(u_0 + \usc) \d y. 
\end{equation}
Here $\G$ is the outgoing fundamental solution of the operator $(\Delta-k^4)$ (or the kernel of the integral operator $(\Delta^2 - k^4 -\rmi 0)^{-1}$).  
The precise formula for the function $\G$ with $k>0$ in $\R^n$ is given as
\begin{align*}
    \G(|x|) &= \frac{\rmi}{8k^2}\left(\frac{|k|}{2\pi|x|}\right)^{\frac{n-2}{2}} \left( H_{\frac{n-2}{2}}^{(1)} (|k||x|)+ \frac{2\rmi}{\pi} K_{\frac{n-2}{2}}(|k||x|) \right) \\
    &=: G_k^H (|x|) + G_k^K(|x|),
\end{align*}
where $H_\nu^{(1)}$ is the Hankel function of the first kind of order $\nu$
and $K_\nu$ is the Macdonald function of order $\nu$.
Throughout this paper, we assume that the perturbation $V$ satisfies the following conditions 
\begin{assumption}\label{a1}
Function $V$ may be expanded with respect to the second argument as
$$V(x,1+s) = V(x,1) + s V^*(x,1) + \frac{s^2}{2}V^{**}(x,s^*),$$
where $V(\cdot,1),V^*(\cdot,1)\in L^p_\mathrm{loc} (\R^n)$, with some $1\le p\le \infty$ and they both meet the decay property 
$$
|V(x,1)|,|V^*(x,1)| \le \frac{C}{|x|^\mu},
$$
when $|x|\ge R$, with some constants $C,R>0$ and $\mu>n$. We assume that function $V^{**}\in L^1_\mathrm{loc}(\R^n)$ and it also satisfies the decay property above, uniformly in $|s^*|\le s$. 
\end{assumption}
The direct scattering problem of finding the unique solution for the Lippmann-Schwinger equation is studied extensively in \cite{HKST2021, KS2022} with the presence of the first order perturbation. From the point of view of this paper, the estimates obtained in those papers are not sufficient and they may be easily improved in the absence of the first order perturbation. Omitting precise proofs, we state the following theorem:
\begin{theorem}\label{t1}
Let function $V$ satisfy Assumption \ref{a1}. Then for all $\rho >0$ there exists $k_0>0$ such that the equation \eqref{LS} has a unique solution in $B_\rho = \left\{ f\in L^\infty(\R^n) :  \|f\|_{L^\infty} \le \rho \right\}$, for all $k\ge k_0$. Moreover, the following norm-estimates are satisfied
\begin{equation}\label{est1}
\|\usc\|_{L^\infty} \le C k^{-\nu},
\end{equation}
with some constant $C$ and for any
\begin{equation*}
\nu < 
\begin{cases}
2, \qquad &n=2\\
\mathrm{min}\{2, \frac{4p - 3}{p}\}, &n=3.
\end{cases}
\end{equation*}
\end{theorem}
\begin{remark}
Banach fixed-point theorem gives us an iterative way of finding the unique solution to the Lippmann-Schwinger equation. Setting $\usc^{(0)} \equiv 0$ and defining the subsequent terms in the sequence via the formula
\begin{equation}\label{eq3}
    \usc^{(j)}(x) = - \I \G(|x-y|) V(y,|u_0+\usc^{(j-1)}| ) (u_0 + \usc^{(j-1)} )\d y,
\end{equation}
we find a sequence that converges to the unique solution. For the rate of convergence we have the following, straightforward estimate
\begin{equation}\label{eq10}
\|\usc - \usc^{(j)}\|_{L^\infty} \le  \frac{C}{k^{\nu}}\|\usc-\usc^{(j-1)}\|_{L^\infty} \le \ldots \le  \frac{C}{k^{j\nu}} \|\usc - 0\|_{L^\infty}\le\frac{C}{k^{(j + 1)\nu}}.
\end{equation}
\end{remark}
For fixed $k\ge k_0$, the function $\usc$, obtained in Theorem \ref{t1}, has the following asymptotical behaviour
\begin{equation*}
    \usc(x,k,\theta) = -\dfrac{\rmi \e^{-\rmi \frac{n-1}{4}\pi}}{4(2\pi)^{\frac{n-1}{2}}}\dfrac{k^{\frac{n-7}{2}}\e^{\rmi k|x|}}{|x|^{\frac{n-1}{2}}} A(k,\theta ',\theta) + o\left(\dfrac{1}{|x|^{\frac{n-1}{2}}} \right),
\end{equation*}
as $|x|\to\infty$.
Here function $A$ is called the scattering amplitude and it is given via the formula
\begin{equation}\label{amplitude}
A(k,\theta',\theta)= \I \e^{-\rmi k (\theta',y)}V(y,|u|)u(y,k,\theta)\d y.
\end{equation}
The scattering amplitude (or far-field pattern) depends on the wavenumber and the direction of the incident wave, as well as on the direction of observation, $\theta'\in\s^{n-1}$. 

In \cite{HKST2021, HKS2022}, the Saito's formula was proved in dimensions $n=2$ and $n=3$, respectively. As a direct consequence of Saito's formula, it follows that the knowledge of full scattering data (scattering amplitude is known in all directions $\theta,\theta'\in\s^{n-1}$ and for arbitrarily large $k>0$) uniquely determines certain combination of first- and zero-order perturbations. In \cite{KS2022}, the backscattering problem was considered. By using Born approximation, it was proved that the main singularities of the same combination of perturbations may be reconstructed from the backscattering data.
From the point of view of fixed angle scattering, we mention the paper by Stefanov \cite{S1992}. There, generic uniqueness of the fixed incident angle scattering problem was proved for Schrödinger equation in $\R^3$. These results were extended for a wider class of potentials by Barcelo et al. in \cite{B&all2020}. In \cite{R2001}, the Schrödinger equation was considered and the method of Born approximation was used to show that the main singularities of a non-smooth potential may be recovered from single observation angle.
We also mention the works \cite{FH2017, FHS2013,S2008, SS2010, S2012, L2011, M2018, MPS2021} where the fixed angle scattering problems are considered in various settings. In particular, the non-linear operators are studied in \cite{FH2017, FHS2013, SS2010, L2011} and the present paper follows the footsteps of \cite{FH2017, FHS2013}.
As an example of higher order operators (and their quasi-linear generalisations) we mention the theory of vibrations of beams and study of elasticity (see \cite{G2010}). 

The following notations are used throughout the text. The symbol 
$L^p_\delta(\R^n)$, \\$1\le p\le\infty, \delta\in\R$ denotes the $p$-based Lebesgue space over $\R^n$ with norm
$$
\| f\|_{L_\delta^p} =\left( \I (1+|x|)^{\delta p} |f(x)|^p \d x\right)^{1/p}.
$$
The weighted Sobolev spaces $W^m_{p,\delta}(\R^n)$ are defined as the spaces of functions whose weak derivatives up to order $m\ge 0$ belong to $L^p_\delta(\R^3)$ and the norm is defined as follows,$$\| f\|_{W_{p,\delta}^m} =\sum_{|\alpha|\le m} \|D^\alpha f\|_{L_\delta^p} .$$ For $L^2$-based space we use the special notation $H^m_\delta(\R^n)=W^m_{2,\delta}(\R^n)$.
Throughout the text the symbol $C$ (compare with the constants $C$ with some special index and special meaning) is used to denoted generic positive constant whose value may change from line to line.

\section{Inverse scattering}
We consider the inverse problem of recovering the potential $V(\cdot,1)$ from the fixed angel scattering data. That is, we assume that for a fixed incident angle $\theta_0\in\s^{n-1}$ the scattering amplitude is known for all possible observation angles $\theta'\in\s^{n-1}$ and for all wave numbers $k\ge k_0$, where $k_0>0$ is as in Theorem \ref{t1}. From now on, we assume that the function $V$ is real-valued.
From the formula \eqref{amplitude}, we have
\begin{align*}
    A(k,\theta',\theta_0)   & = \I \e^{-\rmi k(\theta'-\theta_0,y)} V(y,1) \d y\\ & + \I \e^{-\rmi k(\theta',y)} \big[ \left(V(y,|u|)-V(y,1)\right)u_0(y)+  V(y,|u|)\usc(y) \big]\d y,
\end{align*}
where the latter term can be estimated as $C \|\usc\|_{L^\infty}$.
Thus, for large values of $k>0$, we have
\begin{align*}
    A(k,\theta',\theta_0) \approx \I \e^{-\rmi k(\theta'-\theta_0,y)} V(y,1)\d y = \mathcal{F}\big(V(\cdot,1)\big)(k(\theta'-\theta_0)).
\end{align*}
This heuristic justifies the following definition for the inverse fixed-angle Born approximation
$$
q_B^{\theta_0} (x) = \mathcal{F}^{-1}\left( A(k,\theta',\theta_0) \right) (x),
$$
when we take into account that the inverse Fourier transform is considered in some special coordinates. Indeed, considering $\xi=k(\theta'-\theta_0)$, we have
$$
k= \frac{|\xi|}{2(\hat{\xi},\theta_0)},\quad \theta' = \theta_0 -2(\theta_0,\hat{\xi})\hat{\xi}, \quad \hat{\xi}= \frac{\xi}{|\xi|}
$$
and $\d\xi = \frac{1}{4}|k|^{n-1}|\theta'-\theta_0|^2 \d k\d \theta$. 
Therefore, we give the following definition
\begin{definition}\label{d1}
The inverse fixed-angle Born approximation is given by the formula
\begin{equation*}
    q_B^{\theta_0}(x) := \frac{1}{4(2\pi)^n} \int_{-\infty}^\infty |k|^{n-1}\Is \e^{-\rmi k(\theta_0-\theta',x)}A(k,\theta',\theta_0)|\theta'-\theta_0|^2\d \theta'\d k.
\end{equation*}
\end{definition}
Since Theorem \ref{t1} guarantees the existence of function $\usc$ only when $k\ge k_0$ with some $k_0>0$, the equation \eqref{amplitude} only concerns those values of $k\ge k_0$. For practical reasons, we extend the domain of the scattering amplitude to the whole real-axis by setting $A(k,\theta',\theta_0) = 0$, when $|k|<k_0$ and $A(k,\theta',\theta_0) = \overline{A(-k,\theta',\theta_0)}$, when $k\le -k_0$.
Similarly, we extent the domains of functions $\usc^{(j)}$ (see \eqref{eq3}) to negative values of $k\le -k_0$ by setting
\begin{equation}\label{ext_seq}
    \usc^{(j)}(x) := - \I \widetilde{\G}(|x-y|) V(y,|u_0+\usc^{(j-1)}| ) (u_0 + \usc^{(j-1)} )\d y,
\end{equation}
where $\widetilde{\G}=\G$, when $k>0$, and $\widetilde{\G}= \overline{G^+_{-k}}$, when $k<0$. Similarly, the parts corresponding to the Hankel and Macdonald functions are extended the same way and we have
$\widetilde{\G} = \widetilde{G_k^H} + \widetilde{G_k^K}$.

The reconstruction of jumps and main singularities of the perturbation $V(\cdot,1)$ follows from the main theorem of present paper, which is formulated as
\begin{theorem}[Main theorem]\label{t2}
    Let the real-valued function $V$ satisfy assumption \ref{a1}, with some $2\le p\le\infty$. Then,
    $$
    q_B^{\theta_0}(x) - V(x,1) \in H^t(\R^n) \qquad (\mathrm{mod} \; \Dot{C}(\R^n)),
    $$
    for any $t<\frac{6-n}{2}$.
\end{theorem}
\begin{proof}[Proof of Main theorem]
Following \cite{FHS2013}, we begin the proof by noting that under assumptions \ref{a1}, we have
\begin{align*}
    V(y,|u|)u(y) &= V(y,1)u_0 + \mathcal{V} (y) \usc +  \frac{1}{2} V^*(y,1)u_0^2\ \overline{\usc} + \widetilde{V}(y,s^*) O(|\usc|^2),
\end{align*} 
where $\mathcal{V}(y) = V(y,1) + \frac{1}{2}V^*(y,1)$ and $\widetilde{V}(y,s^*)= V^*(y,1) + V^{**}(y,s^*)$.
By substituting this into \eqref{amplitude}, we obtain
\begin{align*}
    A(k,\theta'&,\theta_0) = \I \e^{-\rmi k(\theta',y)}V(y,|u|)u(y)\d y \\
    &  =\I \e^{-\rmi k(\theta'-\theta_0,y)} V(y,1)\d y - (1-\chi(k)) \I \e^{-\rmi k (\theta'-\theta_0,y)}V(y,1) \d y \\
    &\quad+ \chi(k)\I \e^{-\rmi k(\theta',y)} \left( \mathcal{V}(y)\usc^{(1)} + \frac{1}{2}V^*(y,1) u_0^2\overline{\usc^{(1)}} \right)\d y\\
    &\quad + \chi(k) \I \e^{-\rmi k(\theta',y)}\bigg( \mathcal{V}(y)(\usc-\usc^{(1)})\\ 
    & \qquad \qquad \qquad \qquad + \frac{1}{2}V^*(y,1)u_0^2\overline{(\usc - \usc^{(1)})} + \widetilde{V}(y,s^*) O(|\usc|^2)\bigg)\d y \\
    &  =: A_0(k,\theta')+A_\infty (k,\theta')+A_1 (k,\theta')+A_R(k,\theta'),
\end{align*}
where $\chi$ is the characteristic function of the set $\R \setminus ]-k_0,k_0[$.
Therefore, to prove the theorem, it is enough to consider terms $A_\infty$, $A_1$ and $A_R$ as functions of variable $\mu = k(\theta'-\theta_0)$. 
Clearly, the support of function $A_\infty$ is compact and thus its contribution to the difference $q_B^{\theta_0} - V(\cdot,1)$ is smooth ($C^\infty(\R^n)$). 
Let us next consider the term $A_R$. It follows from the fact that functions $V$ and $V^*$ are $L^1$-integrable, and from the estimates \eqref{est1} and \eqref{eq10}, that there exists a constant $C$ such that
$$
|A_R(k,\theta')|  \le \frac{C}{k^{4}},
$$
uniformly in $|k|\ge k_0$ and $\theta '\in \s^{n-1}$.
Moreover, for the $L_\delta^2$-norms of these terms (with variable $\mu=k(\theta '-\theta_0)$), we have 
\begin{align*}
    &\I (1+|\mu|^2)^\delta |A_R(k,\theta')|^2 \d \mu\\  & \le C\int_{-\infty}^\infty |k|^{n-1}\Is (1+|k(\theta'-\theta_0)|^2)^\delta \dfrac{\chi(k)}{|k|^{8}}\d \theta' \d k \le C \int_{k_0}^\infty k^{n-1+2\delta -8} \d k,
\end{align*}
which is finite for any $\delta < \frac{8-n}{2}$. Hence, 
the contribution of $A_R$ to the difference $q_B^{\theta_0} - V(\cdot,1)$ belongs to the Sobolev space $H^\delta(\R^n)$,
for any $\delta < \frac{8-n}{2}$, as an inverse Fourier transforms of $L_\delta^2$-functions.
The rest of the proof of the main theorem will be given as two separate lemmata.
\begin{lemma}\label{lemma1}
Under the assumptions of Theorem \ref{t2}, the inverse Fourier transform ($\xi = k(\theta'-\theta_0)\rightarrow x $.) of function $A_1$ is of the form
\begin{align}\label{eq11}
    &-\mathcal{F}_{2n}^{-1}\left(\widetilde{\chi}(\xi,\eta) \frac{(\xi+\eta,\theta_0)^3}{|\xi+\eta|^4}\mathrm{P.V.} \frac{\widehat{\mathcal{V}}(\xi)\widehat{V(\cdot,1)}(\eta)}{(|\eta|^2(\xi+\eta) - |\eta+\xi|^2\eta,\theta_0)} \right)(x,x) \notag\\
    &-\mathcal{F}_{2n}^{-1}\left(\widetilde{\chi}(\xi,\eta) \frac{(\xi+\eta,\theta_0)^3}{|\xi+\eta|^4}\mathrm{P.V.} \frac{\widehat{V^*(\cdot,1)}(\xi)\widehat{V(\cdot,1)}(\eta)}{(|\eta|^2(\xi+\eta) + |\eta+\xi|^2\eta,\theta_0)} \right)(x,x) \; (\mathrm{mod}\  H^t(\R^n))\notag\\
    &=: q_1(x) + q_2(x)  \quad (\mathrm{mod} \; H^t(\R^n)),
\end{align}
for any $t<\frac{8-n}{2}$. Here $\mathcal{F}_{2n}^{-1}$ denotes the $2n$-dimensional inverse Fourier transform ($(\xi,\eta)\rightarrow (x,y)$). 
\end{lemma}

\begin{remark}
This lemma allows us to consider $\mathcal{F}^{-1}\left(A_1\right)$ as the first "nonlinear" (quadratic) term in the Born sequence according to Definition \ref{d1} of the inverse fixed-angle Born approximation.
\end{remark}

\begin{proof}[Proof of Lemma \ref{lemma1}]
We split $A_1$ into two parts
\begin{align*}
    A_1(k,\theta') &= \chi(k)\I\!\!\e^{-\rmi k(\theta',y)} \mathcal{V}(y)\usc^{(1)}(y)\d y +\chi(k)\I\!\!\e^{-\rmi k(\theta'-2\theta,y)} V^*(y,1)\overline{\usc^{(1)}(y)}\d y  \\ 
    &= A_1^1 (k,\theta') + A_1^2(k,\theta').
\end{align*}
Substituting \eqref{ext_seq} into $A_1^1$ and splitting the kernel $\widetilde{\G}$ into two parts, one corresponding to the oscillating main part (Hankel function) and the other corresponding to the exponentially decaying part (Macdonald function), yields
\begin{align*}
    &A_1^1(k,\theta')\\ & \quad= \chi(k)\I \e^{-\rmi k(\theta',y)}\mathcal{V}(y) \I \left(\widetilde{G_k^H}(|y-z|) + \widetilde{G_k^K}(|y-z|) \right)V(z,1) \e^{\rmi k(\theta_0,z)} \d z \d y \\
    &\;= A_1^H (k,\theta') + A_1^K(k,\theta').
\end{align*}
Due to the formula $\varphi *\psi = \mathcal{F}^{-1}(\hat{\varphi}\hat{\psi})$, in order to analyse the term $A_1^K$, it is enough to calculate the Fourier transforms of the mappings $z\mapsto \widetilde{G_k^K}(|z|)e^{-\rmi k(\theta_0,z)}$ and $z\mapsto V(z,1)$. Since the Fourier transform of function $G_k^K$ is known, we immediately have
\begin{align*}
    \mathcal{F}\left(\widetilde{G_k^K}(|\cdot|)\e^{-\rmi k(\theta_0,\cdot)}\right)(\xi) = \I \e^{-\rmi (\xi + k\theta_0 ,s)}\widetilde{G_k^K}(|s|) \d s = \frac{1}{2k^2}\frac{1}{|\xi+k\theta_0|^2 + k^2}
\end{align*}
and therefore
\begin{align*}
A_1^K(k,\theta') &= -\frac{\chi(k)}{2k^2}\I \e^{-\rmi k(\theta'-\theta_0,y)}\mathcal{V}(y)\mathcal{F}^{-1} \left( \frac{\widehat{V(\cdot,1)}(\xi)}{|\xi+k\theta_0|^2 + k^2} \right)(y)\d y\\
&= -\frac{\chi(k)}{2k^2} \I\frac{\widehat{\mathcal{V}}(k(\theta'-\theta_0) + \xi)\widehat{V(\cdot,1)}(\xi)}{|\xi+k\theta_0|^2 + k^2}\d \xi.
\end{align*}
Since both $\mathcal{V}$ and $V(\cdot,1)$ are $L^2$--functions, using Hölder inequality and Parseval's identity, we have the following estimate
\begin{align*}
    |A_1^K(k,\theta')| \le C \frac{\chi(k)}{k^4} \|\mathcal{V}\|_{L^2}\|V(\cdot,1)\|_{L^2},
\end{align*}
which holds uniformly in $\theta'\in\s^{n-1}$. Now, by similar calculation that we did for $A_R$, the contribution of $A_1^K$ to the difference $q_B^{\theta_0} - V(\cdot,1)$ belongs to the Sobolev space $H^t(\R^n)$, with any $t<\frac{8-n}{2}$. 
Finally, considering the contribution of $A_1^H$ to the $q_B^{\theta_0}(x)$, we have
\begin{align*}
    \frac{1}{4(2\pi)^n} &\int_{-\infty}^\infty |k|^{n-1} \Is \e^{-\rmi k(\theta_0-\theta',x)}A_1^H(k,\theta')|\theta'-\theta_0|^2\d \theta' \d k \\
    & = \frac{1}{4(2\pi)^n} \int_{-\infty}^\infty |k|^{n-1} \Is \e^{-\rmi k(\theta_0-\theta',x)} \chi(k) \I \e^{-\rmi k(\theta',y)}\mathcal{V}(y) \\ 
    & \qquad  \qquad \qquad \quad \times\I \widetilde{G_k^H}(|y-z|)V(z,1) \e^{\rmi k(\theta_0,z)} \d z \d y|\theta'-\theta_0|^2\d \theta' \d k \\
    & = \frac{1}{4(2\pi)^n} \I\I G_1(x-y,x-z) \mathcal{V}(y)V(z,1)\d y\d z, 
\end{align*}
where 
$$G_1(y,z) = \int_{-\infty}^\infty \chi(k) |k|^{n-1} \Is \widetilde{G_k^H}(|y-z|) \e^{-\rmi k(\theta_0,z) +\rmi k(\theta',y)}|\theta'-\theta_0|^2\d \theta' \d k.$$
Again, by using the formula $\mathcal{F}^{-1}(\hat{\varphi}\hat{\psi}) = \varphi *\psi$, this $2n$--dimensional convolution can be calculated by taking the Fourier transforms of distributions $(y,z)\mapsto G_1(y,z)$ and $(y,z) \mapsto \mathcal{V}(y) V(z,1)$. 
The Fourier transform of $G_k^H$ is known to be
$$
\mathcal{F}\left(G_k^H(|\cdot|)\right) (\xi) = \frac{1}{2k^2} \frac{1}{|\xi|^2 - k^2 -\rmi 0}
$$
and therefore we have
\begin{align*}
&\mathcal{F}_{2n}\left(G_1\right)(\xi,\eta) = \I \I \e^{-\rmi (\xi,y) - \rmi (\eta,z)} G_1(y,z) \d y\d z \\
&=\int_{-\infty}^\infty \!\!\chi(k) |k|^{n-1} \!\!\Is \I \I\!\! \e^{-\rmi (\xi-k\theta',y) - \rmi (\eta+k\theta_0,z)} \widetilde{G_k^H}(|y-z|) \d y\d z |\theta'-\theta_0|^2\d \theta' \d k \\
&= \int_{k_0}^\infty  k^{n-1} \Is \I  \e^{-\rmi (k(\theta_0-\theta') + \xi +\eta,y)}\d y \I \e^{ - \rmi (\eta+k\theta_0,s)} G_k^H(|s|) \d s |\theta'-\theta_0|^2\d \theta' \d k \\
&+ \int_{k_0}^\infty  k^{n-1} \Is \I  \e^{-\rmi (-k(\theta_0-\theta') + \xi +\eta,y)}\d y \I\!\! \e^{ - \rmi (\eta-k\theta_0,s)} \overline{G_k^H(|s|)} \d s |\theta'-\theta_0|^2\d \theta' \d k \\
&= \int_{k_0}^\infty k^{n-1} \Is \I \frac{1}{2k^2}\left(\frac{ \e^{-\rmi (k(\theta_0-\theta') + \xi +\eta,y)}}{|\eta+k\theta_0|^2-k^2-\rmi 0} + \frac{ \e^{-\rmi (-k(\theta_0-\theta') + \xi +\eta,y)}}{|\eta-k\theta_0|^2-k^2+\rmi 0}\right)\d y\\ & \qquad \qquad\qquad \qquad\qquad \qquad\qquad \qquad\qquad \qquad\qquad \qquad  \qquad  \times |\theta'-\theta_0|^2\d \theta' \d k \\
&= \int_{k_0}^\infty k^{n-1} \Is \I \frac{1}{2k^2}\left(\frac{ \e^{-\rmi (k(\theta_0-\theta') + \xi +\eta,y)}}{|\eta|^2+2k(\theta_0,\eta)-\rmi 0} + \frac{ \e^{-\rmi (-k(\theta_0-\theta') + \xi +\eta,y)}}{|\eta|^2-2k(\theta_0,\eta)+\rmi 0}\right)\d y\\ &\qquad \qquad\qquad \qquad\qquad \qquad\qquad \qquad\qquad \qquad\qquad \qquad  \qquad \times |\theta'-\theta_0|^2\d \theta' \d k.
\end{align*}
Considering $\zeta = k(\theta'-\theta_0)$ ($k= \frac{|\zeta|^2}{2(\theta_0,\zeta)}$) and using the fact $\mathcal{F}\left(1\right) = (2\pi)^n\delta$, we have
\begin{align*}
    &\mathcal{F}_{2n}\left(G_1\right)(\xi,\eta) =\!\!\int_\Omega \I \!\!\!\!\frac{2(\zeta,\theta_0)^2}{|\zeta|^4} \left(\frac{ \e^{-\rmi (-\zeta + \xi +\eta,y)}}{|\eta|^2-\frac{|\zeta|^2(\theta_0,\eta)}{(\zeta,\theta_0)}-\rmi 0} + \frac{ \e^{-\rmi (\zeta + \xi +\eta,y)}}{|\eta|^2+\frac{|\zeta|^2(\theta_0,\eta)}{(\zeta,\theta_0)}+\rmi 0} \right)\d y\d \zeta\\
    & \qquad = 2(2\pi)^n \int_\Omega\frac{(\zeta,\theta_0)^2}{|\zeta|^4}  \left(\frac{ \delta(-\zeta + \xi +\eta)}{|\eta|^2-\frac{|\zeta|^2(\theta_0,\eta)}{(\zeta,\theta_0)}-\rmi 0} + \frac{ \delta(\zeta + \xi +\eta)}{|\eta|^2+\frac{|\zeta|^2(\theta_0,\eta)}{(\zeta,\theta_0)}+\rmi 0} \right) \d \zeta \\
    &\qquad = 4(2\pi)^n \widetilde{\chi}(\xi,\eta) \frac{(\xi+\eta,\theta_0)^3}{|\xi+\eta|^4}\mathrm{P.V.} \frac{1}{(|\eta|^2(\xi+\eta) - |\eta+\xi|^2\eta,\theta_0)},
\end{align*}
where the last equality is due to Sokhotski-Plemelj theorem and cancellation of $\delta$-terms therein.
Here, $\Omega = \left\{ \zeta\in\R^n : \left| \frac{|\zeta|^2}{2(\theta_0,\zeta)} \right|\ge k_0 \right\}$ and $\widetilde{\chi}(\xi,\eta) = \chi_\Omega (\xi+\eta),$ where $\chi_\Omega$ is the characteristic function of the set $\Omega.$  
Similarly, for function $A_1^2(k,\theta')$ the following holds
\begin{align*}
    A_1^2(k,\theta') &= \chi(k) \I \e^{-\rmi k(\theta'-2\theta_0,y)} V^*(y,1) \overline{\usc^{(1)}(y)}\d y\\
    & = \chi(k) \I \e^{-\rmi k(\theta'-2\theta_0,y)}V^*(y,1) \I \left(\overline{\widetilde{G_k^H}(|y-z|)}+\overline{\widetilde{G_k^K}(|y-z|)}\right)\\ &\qquad \qquad \qquad \qquad \qquad \qquad \qquad \qquad \qquad \times V(z,1) \e^{-\rmi k(\theta_0,z)}\d z\d y \\
    &= A_2^H(k,\theta') + A_2^K(k,\theta').
\end{align*}
Repeating the estimations from earlier, it is easy to show that the contribution of the latter term to the function $q_B^{\theta_0}$ belongs to the Sobolev space $H^t(\R^n)$, for any $t<\frac{8-n}{2}$.
Again, the contribution of $A_1^2$ to the function $q_B^{\theta_0}$ is of form
\begin{align*}
    \frac{1}{4(2\pi)^n}\I \I G_2(x-y,x-z)V^*(y,1) V(z,1) \d y\d z, 
\end{align*}
where 
$$
G_2(y,z) = \int_{-\infty}^\infty |k|^{n-1} \chi(k) \Is \e^{\rmi k(\theta',y) - \rmi k(\theta_0,2y-z)} \overline{\widetilde{G_k^H}(|y-z|)} |\theta'-\theta_0|^2 \d \theta'\d k.
$$
Taking the $2n$--dimensional Fourier transform of $G_2$ leads us to 
\begin{align*}
    \mathcal{F}_{2n}&\left(G_2\right)(\xi,\eta) = \I \I \e^{-\rmi(\xi,y) - \rmi(\eta,z) } G_2(y,z) \d y\d z \\
    &= \int_{-\infty}^\infty |k|^{n-1} \chi(k) \Is \I \e^{-\rmi(\xi - k\theta' + 2k\theta_0,y)}\\
    &\qquad \qquad \qquad \qquad \times\I \e^{-\rmi(\eta-k\theta_0,z)}\overline{\widetilde{G_k^H}(|y-z|)} \d z\d y|\theta'-\theta_0|^2 \d \theta'\d k\\
    &= \int_{-\infty}^\infty |k|^{n-1} \chi(k) \Is \I \e^{-\rmi(\xi+\eta + k(\theta_0 - \theta'),y)}\\
    &\qquad \qquad \qquad \qquad\times \I \e^{-\rmi(\eta-k\theta_0,s)}\overline{\widetilde{G_k^H}(|s|)} \d z\d y|\theta'-\theta_0|^2 \d \theta'\d k\\
    &= \int_{k_0}^\infty\!\! k^{n-1} \frac{1}{2k^2}\Is \I  \!\!\left(\frac{\e^{-\rmi(\xi+\eta + k(\theta_0 - \theta'),y)}}{|\eta-k\theta_0|^2 - k^2 + \rmi 0}+\frac{\e^{-\rmi(\xi+\eta - k(\theta_0 - \theta'),y)}}{|\eta+k\theta_0|^2 - k^2 - \rmi 0} \right)\d y \\ & \qquad \qquad \qquad \qquad\qquad \qquad \qquad \qquad\qquad \qquad \qquad \qquad \times|\theta'-\theta_0|^2 \d \theta'\d k \\
    & = \int_\Omega \frac{2(\zeta,\theta_0)^2}{|\zeta|^4}\left( \frac{\delta(\xi+\eta + \zeta)}{|\eta|^2 -2\frac{|\zeta|^2(\theta_0,\eta)}{(\zeta,\theta_0)} + \rmi 0} + \frac{\delta(\xi+\eta - \zeta)}{|\eta|^2 +2\frac{|\zeta|^2(\theta_0,\eta)}{(\zeta,\theta_0)} - \rmi 0}\right)\d \zeta\\
    & = 4(2\pi)^n\widetilde{\chi}(\xi,\eta)\frac{(\xi + \eta,\theta_0)^3}{|\xi +\eta|^4} \mathrm{P.V.}\frac{1}{(|\eta|^2(\xi+\eta) + |\xi+\eta|^2 \eta,\theta_0) }.
\end{align*}
The claim of Lemma \ref{lemma1} follows immediately from the $2n$--dimensional Fourier transforms of $G_1$ and $G_2.$ 
\end{proof}
\begin{lemma}\label{lemma2}
Under the assumptions of Theorem \ref{t1}, the first non-linear term in the Born sequence belongs to the Sobolev space $H^t(\R^n)$, with any $t<\frac{6-n}{2}$.
\end{lemma}
\begin{proof}
Due to Lemma \ref{lemma1} it is enough to consider functions $q_1$ and $q_2$ separately (functions $q_1$ and $q_2$ are defined in the Lemma \ref{lemma1}). 
Proceeding as in \cite{S2008}, we assume, without loss of generality, that $\theta_0 = (1,0)$, when $n=2$ and $\theta_0=(1,0,0)$, when $n=3$. Then, for the Fourier transform of the first term in \eqref{eq11}, we have
\begin{align*}
    \mathcal{F}\left(q_1\right) (w) &= -\frac{\widetilde{\chi}(w) }{(2\pi)^n}\frac{w_1^3}{|w|^4} \mathrm{P.V.}\I  \frac{\widehat{\mathcal{V}}(w-\xi)\widehat{V(\cdot,1)}(\xi)}{|\xi|^2 w_1 - |w|^2\xi_1} \d \xi  \\
    & = -\frac{\widetilde{\chi}(w)}{(2\pi)^n}\frac{w_1^2}{|w|^4} \mathrm{P.V.}\I   \frac{\widehat{\mathcal{V}}(w-\xi-a)\widehat{V(\cdot,1)}(\xi+a)}{|\xi|^2 - |a|^2} \d\xi,
\end{align*}
where the $n$--dimensional vector $a$ is given as $(\frac{|w|^2}{2w_1},0)$ or $(\frac{|w|^2}{2w_1},0,0)$, depending on $n=2,3$.

We split the principal value integral into three parts (cf. \cite{OPS2001}, proof of Proposition 3.1.) and consider it outside the singularity at $\frac{|\xi+\eta|^2}{(\xi+\eta,\theta_0)}=0$ and we have
\begin{align*}
    &\left(-\frac{w_1^2}{(2\pi)^n |w|^4}\right)^{-1} \mathcal{F}\left(q_1\right) (w)  = \mathrm{P.V.}\I   \frac{\widehat{\mathcal{V}}(w-\xi-a)\widehat{V(\cdot,1)}(\xi+a)}{|\xi|^2 - |a|^2} \d\xi \\
    & = \left(\int_{|\xi|\le|a|-1} + \mathrm{P.V.}\int_{|a|-1\le|\xi|\le|a|+1} + \int_{|\xi|\ge|a|+1}\right)  \frac{\widehat{\mathcal{V}}(w-\xi-a)\widehat{V(\cdot,1)}(\xi+a)}{|\xi|^2 - |a|^2} \d\xi \\
    & = I_1(w) + I_2(w) + I_3 (w). 
\end{align*}
Terms $I_1$ and $I_3$ can be directly estimated by Hölder inequality to obtain the following estimates
$$
\|I_1\|_{L^\infty}, \; \|I_3\|_{L^\infty} \le \frac{C}{|w|} \|\mathcal{V}\|_{L^2} \|V(\cdot,1)\|_{L^2}
$$
and thus we have that their contributions to $\mathcal{F}\left(q_1\right) (w)$ are locally $L^\infty(\R^n)$ and outside some ball they are $O(|w|^{-3})$.
Now, considering $I_2$, we switch to polar-coordinates and denote the reflection about circle $r=|a|$ by $r^*$ and we have 
\begin{align*}
    &I_2(w) = \lim_{\varepsilon \to 0+} \Big[
    \int_{|a|-1}^{|a|-\varepsilon} r^{n-1} \Is \frac{\widehat{\mathcal{V}}(w-r\omega-a)\widehat{V(\cdot,1)}(r\omega+a)}{r^2 - |a|^2} \d \omega \d r \\
    & \qquad \qquad \qquad + \int_{|a|+\varepsilon}^{|a|+1} r^{n-1} \Is \frac{\widehat{\mathcal{V}}(w-r\omega-a)\widehat{V(\cdot,1)}(r\omega+a)}{r^2 - |a|^2} \d \omega \d r \Big]\\
    & = \lim_{\varepsilon \to 0+}\int_{|a|-1}^{|a|-\varepsilon} \!\!r^{n-1} \Is \!\!\!\!\frac{( \widehat{\mathcal{V}}(w-r\omega-a) - \widehat{\mathcal{V}}(w-r^*\omega-a))\widehat{V(\cdot,1)}(r\omega+a)}{(r - |a|)(r+|a|)} \d \omega \d r\\
    &+\lim_{\varepsilon \to 0+}\int_{|a|-1}^{|a|-\varepsilon} \!\!r^{n-1} \Is \!\!\!\!\frac{ \widehat{\mathcal{V}}(w-r^*\omega-a)(\widehat{V(\cdot,1)}(r\omega+a)-\widehat{V(\cdot,1)}(r^*\omega+a) )}{(r - |a|)(r+|a|)} \d \omega \d r\\
    &+\lim_{\varepsilon \to 0+}\int_{|a|-1}^{|a|-\varepsilon} \Is \!\!\!\!\frac{ \widehat{\mathcal{V}}(w-r^*\omega-a)\widehat{V(\cdot,1)}(r^*\omega+a) }{(r - |a|)}\left[ \frac{r^{n-1}}{r+|a|} - \frac{(r^*)^{n-1}}{r^*+|a|}\right] \d \omega \d r\\
    &= A_1(w) + A_2(w) + A_3(w).
\end{align*}
Since both functions $\mathcal{V}$ and $V(\cdot,1)$ belong to the weighted Lebesgue space $L_1^2(\R^n)$, their Fourier transforms belong to the Sobolev space $H^1(\R^n)$. Therefore, we may estimate terms $A_1$ and $A_2$ by using the inequality of Haj\l asz (see \cite{H1996}) and we have
\begin{align*}
    |A_1(w)| &\le \frac{2}{|w|} \int_{|a|-1\le |\xi|\le |a|} \left(M(\nabla \widehat{\mathcal{V}})((w-\xi-a)) +M(\nabla \widehat{\mathcal{V}}) (w-\xi^*-a)\right)
    \\& \qquad\qquad\qquad\qquad\qquad\qquad\qquad\qquad\qquad\qquad\times|\widehat{V(\cdot,1)}(\xi+a)|\d \xi \\
    & \le \frac{C}{|w|} \|M(\nabla \widehat{\mathcal{V}})\|_{L^2(\R^n)} \|V(\cdot,1)\|_{L^2(\R^n)}
\end{align*}
and similarly,
$$
|A_2(w)| \le \frac{C}{|w|} \|\mathcal{V}\|_{L^2(\R^n)} \| M(\nabla \widehat{V(\cdot,1)}) \|_{L^2(\R^n)}.
$$
Here $M:\ L^p(\R^n) \to L^p(\R^n)$ is the Hardy-Littlewood maximal function. 
For $A_3$ we use the elementary estimate 
$\left|\frac{r^{n-1}}{r+|a|} - \frac{(r^*)^{n-1}}{r^*+|a|}\right| \le C r^{n-1} |r-|a||/|a|$ and conclude that
$$
|A_3(w)| \le \frac{C}{|w|} \|\mathcal{V}\|_{L^2} \|V(\cdot,1)\|_{L^2}.
$$
Now the lemma follows, since we have shown that
$$
\left|\mathcal{F}\left(q_1\right) (w)\right| \le \frac{C}{|w|^3}, 
$$
for large values of $|w|\ge R > 0$. Direct calculation yields that $\mathcal{F}\left(q_1\right)\in L^2_{t}(\R^n)$, with any $t<\frac{6-n}{2}$ and further, $q_1\in H^t(\R^n)$. Repeating this procedure for $q_2$ finishes the proof of Lemma \ref{lemma2}.
\end{proof}
Lemmata \ref{lemma1} and \ref{lemma2} yield the final estimates and thus finish the proof of the main theorem.
\end{proof}
\section*{Acknowledgements}
This work was supported by the Academy of Finland (grant
number 312123, the Centre of Excellence of Inverse Modelling and Imaging 2018-
2025).

\end{document}